\documentclass{article}

\usepackage{amsmath}
\usepackage{amssymb}
\usepackage{latexsym}
\usepackage{color}

\newtheorem{assumption}{Assumption}
\def\qed{ \ \vrule width.2cm height.2cm depth0cm\smallskip}
\newenvironment{proof}{\noindent {\bf Proof.\/}}{$\qed$\vskip 0.1in}

\newcommand{\we}{\wedge}
\newcommand{\ol}{\overline}
\newcommand{\ul}{\underline}

\newcommand{\ba}{\begin{array}}
\newcommand{\ea}{\end{array}}
\newcommand{\be}{\begin{equation}}
\newcommand{\ee}{\end{equation}}
\newcommand{\bea}{\begin{eqnarray}}
\newcommand{\eea}{\end{eqnarray}}
\newcommand{\beaa}{\begin{eqnarray*}}
\newcommand{\eeaa}{\end{eqnarray*}}

\def\dbE{\mathbb{E}}
\def\dbF{\mathbb{F}}

\def\dbH{\mathbb{H}}

\def\dbL{\mathbb{L}}

\def\dbN{\mathbb{N}}
\def\dbP{\mathbb{P}}
\def\dbR{\mathbb{R}}

\def\a{\alpha}
\def\b{\beta}
\def\g{\gamma}
\def\d{\delta}
\def\e{\varepsilon}

\def\k{\kappa}
\def\l{\lambda}

\def\t{\tau}
\def\f{\varphi}
\def\th{\theta}
\def\o{\omega}

\def\Th{\Theta}

\def\O{\Omega}
\def\bfL{\mathbf L}

\def\cA{{\cal A}}

\def\cD{{\cal D}}
\def\cE{{\cal E}}
\def\cF{{\cal F}}

\def\cH{{\cal H}}

\def\cJ{{\cal J}}

\def\cL{{\cal L}}
\def\cM{{\cal M}}
\def\cN{{\cal N}}

\def\cP{{\cal P}}

\def\cT{{\cal T}}

\def\ch{\textsc{h}}

\def\no{\noindent}

\def\ms{\medskip}

\def\q{\quad}

\def\pa{\partial}
\def\cd{\cdot}
\def\cds{\cdots}

\def\qed{ \hfill \vrule width.25cm height.25cm depth0cm\smallskip}

\newcommand{\basa}{\begin{assumption}}
\newcommand{\easa}{\end{assumption}}

\newcommand{\bas}{\begin{assum}}
\newcommand{\eas}{\end{assum}}

\def\limsup{\mathop{\overline{\rm lim}}}
\def\liminf{\mathop{\underline{\rm lim}}}

\def\esup{\mathop{\rm ess\!-\!sup}}
\def\einf{\mathop{\rm ess\!-\!inf}}

\def\pa{\partial}

 \def\cd{\cdot}
\def\cds{\cdots}

\def\sgn{\hbox{\rm sgn$\,$}}

\def\1{{\bf 1}}

\def\:{\!:\!}

\def \proof{{\noindent \bf Proof\quad}}

at 9pt

\begin{document}

\newtheorem{thm}{Theorem}[section]
\newtheorem{lem}[thm]{Lemma}
\newtheorem{cor}[thm]{Corollary}
\newtheorem{prop}[thm]{Proposition}
\newtheorem{rem}[thm]{Remark}
\newtheorem{eg}[thm]{Example}
\newtheorem{defn}[thm]{Definition}
\newtheorem{assum}[thm]{Assumption}

\renewcommand {\theequation}{\arabic{section}.\arabic{equation}}
\def\thesection{\arabic{section}}

\numberwithin{equation}{section}
\numberwithin{thm}{section}

\title{\bf Perron's method for viscosity solutions of semilinear path dependent PDEs}
\author{Zhenjie  {\sc Ren}\footnote{CMAP, Ecole Polytechnique Paris, ren@cmap.polytechnique.fr. Research supported by grants from R\'egion Ile-de-France}   
}
\maketitle

\begin{abstract}
This paper proves the existence of viscosity solutions of path dependent semilinear PDEs via Perron's method, i.e. via showing that the supremum of viscosity subsolutions is a viscosity solution. We use the notion of viscosity solutions introduced in \cite{EKTZ} which considers as test functions all those smooth processes which are tangent in mean. We also provide a comparison result for semicontinuous viscosity solutions, by using a regularization technique. As an interesting byproduct,  we give a new short proof for the optimal stopping problem with semicontinuous obstacles.
\end{abstract}

\section{Introduction}

The recently developed theory of viscosity solutions for path dependent PDEs extends the classical notion of viscosity solution of PDEs introduced by Crandall and Lions \cite{CrandallLions} (for an overview we refer to \cite{CrandallIshiiLions} and \cite{FlemingSoner}). Nonlinear path dependent PDEs appear in various applications, such as the stochastic control of non-Markovian systems \cite{ETZ1} and the corresponding stochastic differential games \cite{PhamZhang}. They are also intimately related to the backward stochastic differential equations introduced by Pardoux and Peng \cite{PardouxPeng}, and their extension to the second order in \cite{CSTV,STZ2}. Loosely speaking, solutions of backward SDEs can be viewed as Sobolev solutions of path-dependent PDEs, and our goal is to develop the alternative notion of viscosity solutions which is well-known to provide a suitable wellposedness and stability theory in the Markovian case $u(t,\o)=u(t,\o_t)$. 

In the recent work of Ren, Touzi and Zhang \cite{RTZ}, the authors focus on the semilinear path dependent PDEs and prove the comparison result for continuous viscosity solutions, in the spirit of the work of Caffarelli and Cabr\'e \cite{CaffarelliCabre} in the context of PDEs.  In \cite{EKTZ,RTZ} it is also proved that the solutions of corresponding backward SDEs are viscosity solutions, instead, we are interested in proving the existence of viscosity solutions to semilinear path dependent PDEs by PDE-type arguments, that is, by Perron's method.  It is worth noting that in the fully nonlinear case, one may no longer depend on backward SDEs for finding viscosity solutions for path dependent PDEs, and thus the Perron method will be necessary. Although our result cannot be applied to the fully nonlinear case directly, many arguments in this paper could be useful. Also, the Perron method is not only useful in proving the existence of viscosity solutions, but also has applications in various contexts, for example, the wellposedness of envelope viscosity solution (see \cite{BardiCapuzzo}), the uniqueness of martingale problems \cite{CK}, etc. In the proof of Perron's method, we follow the same idea as the classical literature on viscosity solutions of PDEs, but the arguments turn out to be different and nontrivial. 

It is well understood in PDE literature that the comparison result for continuous viscosity solutions is not sufficient for the existence of solutions. In Perron's method, we need a comparison result for semicontinuous viscosity solutions. However, the argument in \cite{RTZ} cannot be adapted into our context, because it is not clear whether upper semicontinuous submartingales are almost everywhere punctually differentiable (a crucial intermediate result in \cite{RTZ}). In this paper, we apply a regularization on semicontinuous viscosity solutions so as to mollify them to be continuous. Let $u$ be a viscosity subsolution, and $u^n$ be its regularized version. A reasonable regularization should satisfy:
\beaa
u^n~\mbox{is continuous};\q u^n\rightarrow u, \mbox{as}~ n\rightarrow\infty;
\q u^n~\mbox{is still a viscosity subsolution.}
\eeaa
The regularization we propose involves a backward distance for paths, is new in literature, satisfies all the above conditions and helps to prove the comparison result. It is worth mentioning that a regularization is probably inevitable in the study of the comparison result for fully nonlinear path dependent PDEs. The regularization we find in this paper might shed light on the future research.

As in the previous work on the viscosity solutions of path dependent PDEs, the optimal stopping result plays a crucial role to overcome the non-local-compactness of the path space. Since we treat semicontinuous viscosity solutions in this paper, we need the corresponding result of optimal stopping under nonlinear expectation for semicontinuous obstacles. In the existing literature, Kobylanski and Quenez \cite{KQ} contains the desired result but only in the case of linear expectation. Peng and Xu studied in \cite{PX} reflected backward SDEs with $L^2$ obstacles, and they proved a crucial intermediate result which can lead to the optimal stopping result. However, since their main interest is reflected backward SDEs, there is no direct theorem that we may apply. In this paper, we give a new simple proof for the optimal stopping problem, by using the minimum condition of the Skorokhod decomposition.

The rest of the paper is organized as follows. Section \ref{sec: preliminary} introduces the often used notations. Section \ref{sec: defn vs ppde} recalls the definition of viscosity solutions to path dependent PDEs. Section \ref{sec: main results} presents the main results of the paper: the comparison result of semicontinuous viscosity solutions and the result of Perron's method. Then Section \ref{sec: Perron} contains the proof of Perron's method, and Section \ref{sec: comparison} is devoted to the comparison result. Section \ref{sec: Optimal stopping} reports the proof of the optimal stopping result. Finally, Section \ref{sec: Appendix} is the appendix in which we complete some proofs.

\section{Preliminary}\label{sec: preliminary}

Throughout this paper let $T>0$ be a given finite maturity, $\O:=\{\o\in C([0,T];\dbR^d):\o_0=0\}$ be the set of continuous paths starting from the origin, and $\Theta:=[0,T]\times\O$. We denote $B$ as the canonical process on $\O$, $\dbF = \{\cF_t, 0\le t\le T\}$ as the canonical filtration,  $\cT$ as the set of all $\dbF$-stopping times taking values in $[0,T]$. Further let $\cT^+$ denote the subset of $\t\in\cT$  taking values in $(0,T]$, and for $\ch \in \cT$, let $\cT_\ch$ and $\cT_\ch^+$ be the subset of $\t\in \cT$ taking values in $[0, \ch]$ and in $(0, \ch]$, respectively. We also denote $\dbP_0$ as the Wiener measure on $\O$, and define the augmented filtration by $\dbF^*:=\{\cF_t\vee \cN; 0\le t\le T\}$, where $N$ is the collection of all $\dbP_0$-null sets.

Following Dupire \cite{Dupire}, we introduce the following pseudo-distance on $\Theta$:
 \beaa
\|\o\| := \sup_{0\le s\le T} |\o_s|,\q  d(\th,\th')
 :=
 |t-t'|+\|\o_{t\wedge}-\o'_{t'\wedge}\|
 &\mbox{for all}&
 \th=(t,\o), \th'=(t',\o')\in\Theta.
 \eeaa
We say  a process valued in some metric space $E$ is in $C^0(\Theta,E)$ whenever it is continuous with respect to $d$. Similarly, $\dbL^0(\cF, E)$ and $\dbL^0(\dbF, E)$ denote the set of $\cF$-measurable random variables and $\dbF$-progressively measurable processes, respectively. We remark that $C^0(\Theta,E)\subset \dbL^0(\dbF, E)$, and when $E=\dbR$, we shall omit it in these notations. 

In this paper, we also use another (backward) pseudo-distance on $\Th$:
\beaa
\overleftarrow{d}(\th,\th'):=|t-t'|+\sup_{s\ge 0}|\o_{(t-s)\vee 0}-\o'_{(t'-s)\vee 0}|.
\eeaa
The following lemma explains the relation between $d(\cd,\cd)$ and $\overleftarrow{d}(\cd,\cd)$.

\begin{lem}\label{lem: difference d,d}
For all $\th,\th'\in\Th$, we have
\be\label{difference d,d}
\big|d(\th,\th')-\overleftarrow{d}(\th,\th')\big|\le \bar\rho(\th,|t-t'|), 
\q \mbox{where}\q \bar\rho(\th,\d):=\sup_{|s-s'|\le \d}|\o_{t\we s}-\o_{t\we s'}|.
\ee
In particular, a function $f:\Th\rightarrow\dbR$ is continuous in $d(\cd,\cd)$ if and only if $f$ is continuous in $\overleftarrow{d}(\cd,\cd)$.
\end{lem}
\begin{proof}
Define $\o_s=0$ for $s<0$. The first claim follows from the simple observation:
\beaa
\Big||\o_{t\we s}-\o'_{t'\we s}|-|\o_{t\we (t-t'+s)}-\o'_{t'\we s}|\Big|\le
|\o_{t\we s}-\o_{t\we (t-t'+s)}|\le
 \bar\rho(\th,|t-t'|).
\eeaa
The second claim is a trivial corollary.
\end{proof}

For any $A\in \cF_T$, $\xi\in \dbL^0(\cF_T, E)$, $X\in \dbL^0(\dbF,E)$, and  $(t,\o)\in \Th$, define:
 \beaa
&A^{t,\o} := \{\o'\in \O: \o\otimes_t \o' \in A\},\q  \xi^{t,\o}(\o')
 :=
 \xi(\o\otimes_t\o'), 
\q  X^{t,\o}_s(\o')
 :=
 X(t+s,\o\otimes_t\o')&\\
 &~~\mbox{for all}~~\o'\in\O,~~
 \mbox{where}\q
 (\o\otimes_t\o')_s:=\o_s\1_{[0, t]}(s) +(\o_t+\o'_{s-t})\1_{(t, T]} (s),\q 0\le s\le T.&
 \eeaa
 Following the standard arguments of monotone class, we have the following simple results.

\begin{lem}\label{lem:measurability F}
Let $0\le t\le s\le T$ and $\o\in \O$. Then
 $A^{t,\o}\in \cF_{s-t}$ for all $A\in\cF_{s}$, $\xi^{t,\o} \in \dbL^0(\cF_{s-t}, E)$ for all $\xi\in \dbL^0(\cF_s, E)$, $X^{t,\o} \in \dbL^0(\dbF, E)$ for all  $X\in \dbL^0(\dbF, E)$, and $\t^{t,\o}-t\in \cT_{s-t}$ for all $\t\in\cT_s$.
\end{lem}

We consider the semilinear path dependent PDE:
\be\label{PPDE}
-\cL u(\th) - F(\cd,u,\pa_\o u)(\th)=0,\q\mbox{where}~\cL u:=\pa_t u+\frac12\pa_{\o\o}^2 u.
\ee
Introduce a family of probability measure:
\beaa
\cP:=\Big\{\dbP_\mu: \frac{d\dbP_\mu}{d\dbP_0}=\exp{\big(\int_0^T \mu_t dB_t-\frac12\int_0^T \mu^2_t dt\big)},~\mbox{for some}~\mu\in\dbL^0(\dbF,\dbR^d),\|\mu\|\le L\Big\},
\eeaa
where $L$ is a constant. Then the corresponding nonlinear expectations are defined as:
\beaa
\ol\cE_L[\cd]:=\sup_{\dbP\in\cP}\dbE^\dbP[\cd],
\q
\ul\cE_L[\cd]:=\inf_{\dbP\in\cP}\dbE^\dbP[\cd].
\eeaa

\section{Definition of viscosity solution for path dependent PDEs}\label{sec: defn vs ppde}

As showed in \cite{RTZ-survey,RTZ}, we may define viscosity solutions via semijets. Define the following space of measurable processes:
\beaa
\bfL^1(\cP_L):=\Big\{X\in\dbL^0(\dbF): \ol\cE_L\big[\sup_{s\le T-t}|X^{\th}_s|\big]<\infty~\mbox{for all}~\th=(t,\o)\in\Th\Big\}.
\eeaa

\begin{defn}[Semijets]\label{Def jet}
For $u\in \bfL^1(\cP_L)$, the subjet and superjet of $u$ at $\th$ are defined as:
 \beaa
 \ul\cJ_L u(\th)
 &:=&
 \big\{(\a,\b)\in \dbR\times\dbR^d:~u(\th)=\max_{\t\in\cT_\ch}\ol\cE_L[u^\th_\t-\a\t-\b B_\t],~~\mbox{for some}~\ch\in\cT^+
 \big\};\\
   \ol\cJ_L u(\th)
 &:=&
 \big\{(\a,\b)\in \dbR\times\dbR^d:~u(\th)=\min_{\t\in\cT_\ch}\ul\cE_L[u^\th_\t-\a\t-\b B_\t],~~\mbox{for some}~\ch\in\cT^+
 \big\}.
 \eeaa
\end{defn}

\begin{defn}[Viscosity solution]
Let $u\in\bfL^1(\cP_L)$. Then,

\no {\rm (i)}\q  $u$ is a viscosity subsolution (resp. supersolution) of the path dependent PDE \eqref{PPDE}, if for all $\th\in\Th$, and $(\a,\b)\in\ul\cJ_L u(\th)$ (resp. $\ol\cJ_L u(\th)$), it holds that
\beaa
-\a- F(\th,u(\th),\b)~\le~\mbox{(resp. $\ge$)}~0.
\eeaa

\no {\rm (ii)}\q  $u$ is a viscosity solution of the path dependent PDE \eqref{PPDE}, if $u$ is both a viscosity subsolution and a viscosity supersolution.
\end{defn}

\begin{rem}{\rm
The definition of viscosity solutions depends on the constant $L$. In \cite{RTZ}, the authors give the name as $\cP_L$-viscosity sub-/super-solutions. For the simplification of notations, we simply call them viscosity sub-/super-solutions in this paper.
}
\end{rem}

\section{Main results}\label{sec: main results}

\subsection{Comparison result for semicontinuous viscosity solutions}

In \cite{RTZ}, a comparison result is proved for continuous viscosity solutions. In this paper, we provide an extension to semicontinuous viscosity solutions, which plays an important role in the Perron approach.

\begin{assum}\label{assum: F}
The generator function $F(\th,y,z)$ satisfies the following assumptions.

\no {\rm (i)}\q  $F$ is uniformly Lipschitz continuous in $(y,z)$, i.e. there exists a constant $L$ such that
\beaa
|F(\cd,y,z)-F(\cd,y',z')|\le L|y-y'|+L|z-z'|.
\eeaa

\no {\rm (ii)}\q There exists $F^0\in C^0(\Th)$ such that $|F(\cd,0,0)|\le F^0$.

\no {\rm (iii)}\q There exists a function $\rho^F:(\th,x,y)\in \Th \times \dbR\times \dbR \longrightarrow\dbR$ such that $\rho^F$ is continuous in $(\th,x,y)$ and non-decreasing in $\g$, $\rho^F(\th,0,y)=0$ for all $(\th,y)\in \Th\times \dbR$, and
\beaa
|F(\th,y,\cd)-F(\th',y,\cd)|\le \rho^F\big(\th,d(\th,\th'),y\big),\q \mbox{for all}~~\th,\th'\in\Th.
\eeaa
\end{assum}

Our comparison result is based on the following consequence of  Theorem 4.1 in Ren, Touzi and Zhang \cite{RTZ}.

\begin{thm}\label{thm: comparison continuous}
Let Assumption \ref{assum: F} {\rm (i)} and {\rm(ii)} hold true, and $u, v\in \mbox{C}^0(\Th)$ be bounded viscosity subsolution and supersolution of path dependent PDE \eqref{PPDE}, respectively. If $u_T\le v_T$, then $u\le v$ on $\Th$.
\end{thm}

\begin{rem}{\rm
The comparison result in \cite{RTZ} is established fo continuous viscosity subsolutions and supersolutions which do not need to be bounded, but satisfy some integrability condition. In Assumption \ref{assum: F}, the first two are the same as the assumption in \cite{RTZ}, while (iii) is the extra assumption for the comparison result of semicontinuous viscosity solutions.
}
\end{rem}

\begin{defn}
A function $u:\Th\rightarrow\dbR$ belongs to $\mbox{\rm USC}_b$ (resp. $\mbox{\rm LSC}_b$), if $u$ is bounded and satisfies 
\beaa
u(\th) ~\ge~ \limsup_{d(\th,\th')\rightarrow 0}u(\th')\q\mbox{(resp.} ~\le~ \liminf_{d(\th,\th')\rightarrow 0} u(\th')\mbox{).}
\eeaa
\end{defn}

We will prove in Section \ref{sec: comparison} that

\begin{thm}\label{thm: comparison}
Let Assumption \ref{assum: F} hold true, and $u\in \mbox{\rm USC}_b(\Th), v\in \mbox{\rm LSC}_b(\Th)$ be viscosity subsolution and supersolution of path dependent PDE \eqref{PPDE}, respectively. If $u_T\le v_T$, then $u\le v$ on $\Th$.
\end{thm}

\begin{rem}\label{rem:regularization}{\rm
The argument of proving the comparison result for continuous viscosity solutions in \cite{RTZ} cannot be adapted directly to our context, because it is not clear whether a USC submartingale is almost everywhere punctually differentiable (see the definition in \cite{RTZ}). Our strategy is to apply a regularization so as to introduce continuous approximations which are still viscosity sub-/super-solutions, and then we apply the comparison result for continuous viscosity solutions. Let $u$ be a viscosity subsolution, and $u^n$ be its regularized version. A reasonable regularization should satisfy:
\beaa
u^n~\mbox{is continuous};\q u^n\rightarrow u, \mbox{as}~ n\rightarrow\infty;
\q u^n~\mbox{is still a viscosity subsolution.}
\eeaa
The regularization introduced in Section \ref{sec: regularization} satisfies all above conditions, and helps to prove the comparison result.
}
\end{rem}

\subsection{Existence via Perron's method}

Due to Proposition 3.14 in \cite{ETZ1}, we may equivalently study the existence of viscosity solution for the equation corresponding to the change of variable: $\tilde u_t:=e^{-L t}u_t$. It follows from the Lipschitz property of the nonlinearity $F$ in $y$ that we may assume without loss of generality that $F$ is increasing in $y$.

\begin{assum}\label{assum: Perron}
The generator function $F(\th,y,z)$ satisfies (i) of Assumptions \ref{assum: F} and:

\no {\rm (i)}\q $F$ is continuous in $\th$.

\no {\rm (ii)}\q $F$ is non-decreasing in $y$.
\end{assum}

For a function $w$ on $\Th$, we define its USC and LSC envelops:
\beaa
w^*(\th):=\limsup_{d(\th,\th')\rightarrow 0} w(\th')
\q\mbox{and}\q
w_*(\th):=\liminf_{d(\th,\th')\rightarrow 0} w(\th').
\eeaa
We will prove in Section \ref{sec: Perron} that:

\begin{thm}\label{thm: Perron}
Let Assumption \ref{assum: Perron} and the comparison result of Theorem \ref{thm: comparison} hold true. Assume further that there is a viscosity subsolution $\ul u\in \mbox{\rm USC}_b(\Th)$ and a supersolution $\ol v\in \mbox{\rm LSC}_b(\Th)$ of Equation \eqref{PPDE} which satisfy the boundary condition $(\ul u_*)_T = \ol v^*_T =\xi$. Denote
\beaa
\cD := \big\{\phi: \phi\in\mbox{\rm USC}_b(\Th)~\mbox{is a viscosity suboslution of Equation \eqref{PPDE} and}~~\ul u\le \phi\le \ol v \big\}.
\eeaa
Then $u(\th):=\sup\{\phi(\th):\phi\in\cD\}$ 
is a continuous viscosity solution of Equation \eqref{PPDE}, and satisfies the boundary condition $u_T=\xi$.
\end{thm}

\section{Perron's method}\label{sec: Perron}

We will prove in the following subsections the two propositions:
\begin{prop}\label{prop: subsol}
$u^*\in \mbox{\rm USC}_b(\Th)$ is a viscosity subsolution of Equation \eqref{PPDE}.
\end{prop}

\begin{prop}\label{prop: supersol}
$u_*\in \mbox{\rm LSC}_b (\Th)$ is a viscosity supersolution of Equation \eqref{PPDE}.
\end{prop}

Then the comparison result allows to complete the proof.

\ms
\no {\bf Proof of Theorem \ref{thm: Perron}}\q 
Since $u\ge \ul u$, we have $u_*\ge \ul u_*$, in particular, $(u_*)_T\ge \xi$. On the other hand, since $u\le \ol v$, we have $u^*\le \ol v^*$, in particular, $u^*_T\le \xi$. Therefore, $u^*_T\le (u_*)_T$, and it follows from the comparison result that $u^*\le u_*$. We conclude that $u^*=u=u_*$, and thus $u$ is a bounded continuous viscosity solution of Equation \eqref{PPDE}.
\qed

\subsection{Some useful lemmas}

As in \cite{EKTZ,RTZ-survey,RTZ}, the optimal stopping result is crucial in the current theory of viscosity solution to path dependent PDE. As we are going to treat semicontinuous viscosity solutions, we need an optimal stopping result for semicontinuous obstacles.  

\begin{defn}
{\rm (i)}\q A random variable $X$ is $\ol\cE_L$-uniformly integrable if
$$\lim_{A\rightarrow\infty}\ol\cE_L\big[|X|;|X|\ge A\big]\rightarrow 0.$$

\no {\rm (ii)}\q  A family of random variables $\{X_\a\}$ is $\ol\cE_L$-uniformly integrable if
$$\lim_{A\rightarrow\infty}\sup_{\a}\ol\cE_L\big[|X_\a|;|X_\a|\ge A\big]\rightarrow 0.$$
\end{defn}

One may easily prove the following two lemmas.

\begin{lem}[Dominated convergence]\label{lem: dominate_conv}
Let $X_n$ be a sequence of r.v.'s such that $\{X_n\}_n$ is $\ol\cE_L$-uniformly integrable, and $X_n\rightarrow 0$, $\dbP_0$-a.s. Then we have $\lim_{n\rightarrow\infty}\ol\cE_L\big[|X_n|\big]=0$.
\end{lem}

\begin{lem}[Fatou's lemma]\label{lem: Fatou}
Let $X_n$ be a sequence of bounded r.v.'s. Then we have
\beaa
\limsup_{n\rightarrow\infty} \ol\cE_L\big[X_n\big]
&\le & \ol\cE_L\big[\limsup_{n\rightarrow\infty} X_n\big].
\eeaa
\end{lem}

Denote by $\cT_*$ the set of all $\dbF^*$-stopping times

\begin{thm}[Optimal stopping for semicontinuous obstacle]\label{thm: OS}
Let $X$ be an $\dbF^*$-progressively measurable process such that
 
{\rm (i)}\q $X$ is upper semicontinuous (u.s.c.) in $t$, $\dbP_0$-a.s.;

{\rm (ii)}\q $\sup_t X^+_t$ is $\ol\cE_L$-uniformly integrable;

{\rm (iii)}\q $X^-_t$ is $\ol\cE_L$-uniformly integrable for all $t\in [0,T]$. 

\no Define 
\be\label{eq:OS}
Y(\th) = \sup_{\t\in\cT_*}\ol\cE_L[X^\th_\t].
\ee
Then there exits a stopping time $\t^*\in\cT_*$ such that $Y_0=\ol\cE_L[X_{\t^*}]$ and $X_{\t^*}=Y_{\t^*}$, $\dbP_0$-a.s.
\end{thm}
This theorem will be proved in Section \ref{sec: Optimal stopping}. Based on Theorem \ref{thm: OS}, we may prove the following lemma similar to Lemma 4.9 in \cite{RTZ}, but concerning pathwise u.s.c. functions.

\begin{lem}\label{lem: fundamental}
Let $u_{\cdot\we\ch}$ satisfy the assumptions in Theorem \ref{thm: OS} and assume that $u_0 > \ol\cE_L[u_{\ch}]$ for some  $\ch\in\cT^+$. Then there exists $\o^* \in \O$ and $t^*<\ch(\o^*)$ such that $(0,0)\in\underline\cJ_Lu(\th^*)$.
\end{lem}

\proof Define the optimal stopping problem $Y$ by \eqref{eq:OS} with $X := u_{\cd\we\ch}$. Let $\t^*\in \cT_*$ be the optimal stopping rule. By Theorem \ref{thm: OS} we have
\beaa
\ol \cE_L[u_{\t^*}] = Y_0 \ge u_0 > \ol \cE_L [ u_\ch] 
&\mbox{and}& 
\dbP_0\big[u_{\t^*} = Y_{\t^*}\big] = 1,  
\eeaa
and it follows that $\dbP_0\big[u_{\t^*} = Y_{\t^*}, \t^* <\ch\big] > 0$. Then there exists $\o^* \in \O$ such that $t^* := \t^*(\o^*) < \ch(\o^*)$ and $u_{t^*}(\o^*) = Y_{t^*}(\o^*)$.  Therefore, $(t^*, \o^*)$ is the desired point.
\qed

The next result about semijet is useful for proving stability results.

\begin{lem}\label{lem: stability}
Let $u\in \dbL^0(\dbF)$ be bounded, and  $u^n\in \dbL^0(\dbF)$ be bounded and  pathwise u.s.c. $\dbP_0$-a.s. Fix $\th\in \Th$, and suppose that

\no {\rm (i)}\q there exists a sequence $\{\th^n\}\subset \Th$ such that 
\beaa
d(\th^n,\th)\rightarrow 0\q\mbox{and}\q u(\th)=\lim_{n\rightarrow\infty} u^n(\th^n);
\eeaa 

\no {\rm (ii)}\q for any $\bar\th\in\Th$ and any sequence $\{\bar\th^n\}\subset\Th$ such that $d(\bar\th^n,\bar\th)\rightarrow 0$, it holds
\beaa
u(\bar\th)\ge \limsup_{n\rightarrow\infty} u^n(\bar\th^n).
\eeaa
Then, for any $(\a,\b)\in \ul\cJ_L u(\th)$, $n\in\dbN$ and $\e>0$, there exits $N_n\ge n$ and $\hat\th^n$ such that
\beaa
d(\th^{N_n},\hat\th^n)\le \frac{1}{n}\q\mbox{and}\q
(\a+\e,\b)\in {\rm cl}\Big(\ul\cJ_L u^{N_n}(\hat\th^n)\Big).
\eeaa
\end{lem}
\begin{proof}
Since $(\a,\b)\in\ul\cJ_L u(\th)$, there exists $\ch\in\cT^+$ such that $u(\th)=\max_{\t\in\cT_\ch}\ol\cE_L[u^\th_\t-\a\t-\b B_\t]$. Denote $\ch^n(\o'):=\ch(\o')\we \inf\{t':\|\o'_{t'}\|>\frac{1}{n}\}$. Then for any $\e>0$, it holds
\beaa
u(\th)>\ol\cE_L[u^\th_{\ch^n}-(\a+\e)\ch^n-\b B_{\ch^n}].
\eeaa
Further, by (i) and (ii), we obtain
\beaa
\lim_{m\rightarrow\infty}u^m(\th^m)
>\ol\cE_L\big[\limsup_{m\rightarrow\infty}(u^m)^{\th^m}_{\ch^n}-(\a+\e)\ch^n-\b B_{\ch^n}\big]
\ge \limsup_{m\rightarrow\infty}\ol\cE_L\big[(u^m)^{\th^m}_{\ch^n}-(\a+\e)\ch^n-\b B_{\ch^n}\big].
\eeaa
Therefore, for each $n$, there exists $N_n\ge n$ such that
\beaa
u^{N_n}(\th^{N_n})>\ol\cE_L\big[(u^{N_n})^{\th^{N_n}}_{\ch^n}-(\a+\e)\ch^n-\b B_{\ch^n}\big].
\eeaa
Then, by Lemma \ref{lem: fundamental}, we may find $\hat\th^n$ such that
\beaa
d(\hat\th^n,\th^{N_n})\le \frac{1}{n}
\q\mbox{and}\q
(\a+\e,\b)\in \ul\cJ_L u^{N_n}(\hat\th^n).
\eeaa
\end{proof}

To finish this subsection, we study a special path dependent PDE, and give one of its viscosity solutions by a stochastic representation. Let $u$ be a bounded process and $\ch\in \cH$, and define a function:
\beaa
\eta(\th):=\ul\cE_L\big[(u_\ch)^\th-\a(\ch^\th-t)-\b B_{\ch^\th-t}\big].
\eeaa 

\begin{prop}\label{prop: especial sol}
{\rm (i)}\q  $\eta$ is a viscosity subsolution of the path dependent PDE:
\beaa
-\cL \eta(\th) +\a + L|\b-\pa_\o \eta(\th)| &=& 0.
\eeaa

\no {\rm (ii)}\q  If $u$ is Lipschitz continuous, then $\eta$ is continuous on $\{\th: t \le \ch(\o)\}$.
\end{prop}
We will report the proof in the appendix.

\subsection{Equivalent definitions of viscosity solution}

Denote by $\cH$ the collection of all the stopping times of the form of $\ch^{s,O}:=\inf\big\{t\ge 0:\th\notin [0,s)\times O\big\}\in\cT^+$, where $s>0$ and $O\subset \dbR^d$ is a bounded open convex set.

\begin{defn}\label{Def jet constant}
For $u\in \bfL^1(\cP_L)$, we define for each $\th\in\Th$:
 \beaa
 \ul\cJ'_L u(\th)
 &:=&
 \big\{(\a,\b)\in \dbR\times\dbR^d:~u(\th)=\max_{\t\in\cT_\ch}\ol\cE_L[u^\th_\t-\a\t-\b B_\t],~~\mbox{for some}~\ch\in\cH
 \big\};\\
   \ol\cJ'_L u(\th)
 &:=&
 \big\{(\a,\b)\in \dbR\times\dbR^d:~u(\th)=\min_{\t\in\cT_\ch}\ul\cE_L[u^\th_\t-\a\t-\b B_\t],~~\mbox{for some}~\ch\in\cH
 \big\}.
 \eeaa
\end{defn}
Comparing to Definition \ref{Def jet}, we replace the stopping time $\ch\in \cT^+$ by a hitting time in $\cH$.

\begin{prop}\label{prop: equiv_defn}
Suppose that $u\in \mbox{\rm USC}_b(\Th)$ and that the generator $F:(\th,y,z)\mapsto\dbR$ satisfies Assumption \ref{assum: Perron}. Then $u$ is a viscosity subsolution of Equation \eqref{PPDE} if and only if 
\be\label{eq: equivalent def}
-\a-F(\th,u(\th),\b)\le 0,\q\mbox{for all}\q \th\in \Th, ~(\a,\b)\in \ul\cJ'_L u(\th).
\ee
The similar result holds for supersolutions.
\end{prop}
\begin{proof}
The 'only if' part is trivial by the definitions. We will only prove the 'if' part. Fix a $\th\in \Th$, and suppose $(\a,\b)\in \ul\cJ_L u(\th)$, i.e.
\beaa
u(\th) = \max_{\t\in\cT_\ch}\ol\cE_L[u^\th_\t-\a\t-B_\t]\q\mbox{for some}~~\ch\in\cT^+.
\eeaa
For any $\d>0$, we may suppose $\ch<\hat\ch_\d:=\inf\{t': d(\th',0)\ge \d \}$. Then for any $\e>0$ it holds
\beaa
u(\th) ~>~ \ol\cE_L[u^\th_\ch - (\a+\e)\ch-\b B_\ch].
\eeaa We next define a sequence of hitting time:
\beaa
\bar\ch^n_0:=0,~~ \bar\ch^n_{k+1} := \big(\bar\ch^n_k+\frac{1}{n}\big) \we \inf\big\{t'\ge \bar\ch^n_k: |\o'_{t'}-\o'_{\bar\ch^n_k}|\ge \frac{1}{n} \big\},~~\mbox{for all}~~k\ge 0,
\eeaa
and define $\ch_n:=\inf\{\bar\ch^n_k: \bar\ch^n_k>\ch\}$. Clearly $\ch_n\downarrow\ch$. Since $u\in \mbox{USC}_b(\Th)$, it follows from Fatou's Lemma (Lemma \ref{lem: Fatou}) that
\beaa
\ol\cE_L[u^\th_\ch-(\a+\e)\ch-B_\ch]\ge \ol\cE_L\big[\limsup_{n\rightarrow\infty}(u^\th_{\ch_n}-(\a+\e)\ch_n-B_{\ch_n})\big]
\ge \limsup_{n\rightarrow\infty}\ol\cE_L[u^\th_{\ch_n}-(\a+\e)\ch_n-B_{\ch_n}].
\eeaa 
So there exists $n$ sufficiently large such that
\beaa
u(\th)>\ol\cE_L[u^\th_{\ch_n}-(\a+\e)\ch_n-B_{\ch_n}].
\eeaa
By Lemma \ref{lem: fundamental}, there exists $\th^*\in\Th$ such that $t^*<\ch_n(\o^*)$ and
\beaa
u(\th^*)=\max_{\t\in \cT_{\ch_n^{\th^*}}} \ol\cE_L[u^{\th^*}_\t-(\a+\e)\t-B_\t].
\eeaa
Note that if $\bar\ch^n_k(\th^*)\le t^*<\bar\ch^n_{k+1}(\th^*)$, then $\ch^{\th^*}_n-t^* \ge \ch^*:= (\bar\ch^n_{k+1})^{\th^*}-t^*\in\cH$. It follows that
\beaa
u(\th^*)=\max_{\t\in \cT_{\ch^*}} \ol\cE_L[u^{\th^*}_\t-(\a+\e)\t-B_\t].
\eeaa
By \eqref{eq: equivalent def}, we obtain that
\beaa
-(\a+\e)-F(\th^*,u(\th^*),\b)\le 0.
\eeaa
Finally, by letting $\d,\e\rightarrow 0$ and $n\rightarrow\infty$, we obtain:  $-\a -F(\th, u(\th),\b)\le 0$.
\end{proof}

\subsection{Subsolution property}

\no {\bf Proof of Proposition \ref{prop: subsol}}\q
Fix any $\th\in\Th$. By the definition of $u$ and $u^*$, there is a sequence of functions $\{\phi^n\}\subset \cD$ and a sequence $\{\th^n\}\subset\Th$ such that
\beaa
d(\th^n,\th)\rightarrow 0 \q\mbox{and}\q u^*(\th) ~=~ \lim_{n\rightarrow\infty} \phi^n(\th^n). 
\eeaa
Then by Lemma \ref{lem: stability}, for any $(\a,\b)\in \ul\cJ_L u(\th)$, $n\in\dbN$ and $\e>0$, there is $N_n\ge n$ and $\hat\th^n$ such that
\beaa
d(\th^{N_n},\hat\th^n)\le \frac{1}{n}\q\mbox{and}\q
(\a+\e,\b)\in {\rm cl}\Big(\ul\cJ_L \phi^{N_n}(\hat\th^n)\Big).
\eeaa
Further, since $\phi^n$ ($\le u$) is a viscosity subsolution of Equation \eqref{PPDE} for each $n$, we deduce from the non-decrease of $F$ in $y$ that
\beaa
-(\a+\e)-F(\hat\th^n,u(\hat\th^n),\b)
\le -(\a+\e)-F(\hat\th^n,\phi^{N_n}(\hat\th^n),\b)\le 0.
\eeaa
Then since $\limsup_{n\rightarrow\infty} u(\hat\th^n)\le u^*(\th)$, by letting $n\rightarrow\infty$ we obtain that
\beaa
-(\a+\e)-F(\th,u^*(\th),\b)\le 0.
\eeaa
Finally, by letting $\e\rightarrow 0$, we get the desired result.
\qed

\begin{prop}
It holds that $u=u^*\in \mbox{\rm USC}_b(\Th)$ is a viscosity subsolution of Equation \eqref{PPDE}.
\end{prop}
\begin{proof}
By the previous proposition, we know that $u^*\in \cD$, and thus $u^*\le u$. On the other hand, by the definition of $u^*$, it holds that $u^*\ge u$. Therefore, $u=u^*$.
\end{proof}

\subsection{Supersolution property}

\no {\bf Proof of Proposition \ref{prop: supersol}}\q
\no {\bf 1.}\q Suppose that $u_*$ is not a viscosity supersolution. Then by Proposition \ref{prop: equiv_defn}, there is $\th^0=(t^0,\o^0)\in \Th$ and $(\a,\b)\in \ol\cJ'_L u_*(\th^0)$, i.e. $u_*(\th^0) =\min_{\t\in\cT_\ch} \ul\cE_L [(u_*)^{\th^0}_\t -\a \t - \b B_\t]$ for some $\ch\in\cH$, such that
\be\label{eq: u* delta}
-\a - F(\th^0,u_*(\th^0),\b) ~=:~ -2\d ~<~ 0 .
\ee
Since $F(\th,y,z)$ is non-decreasing in $y$ and $u_*\in\mbox{LSC}_b(\Th)$, it follows from \eqref{eq: u* delta} that
\be\label{eq: u delta neighbor}
-\a + \d - F(\cd,u_*,\b)  < 0 \q\mbox{on}~~O_{9\e_0}:=\{\th: d(\th^0,\th)<9\e_0\}~~\mbox{for some small}~\e_0>0. 
\ee
Without loss of generality, we may assume that $\ch$ is in the form of:
\beaa
\ch(\o) = 3\e_1\we\inf\{s:|\o_s|\ge 3\e_1 \}\q\mbox{for some}~\e_1>0~\mbox{such that}~3\e_1< 3\e_0\we\rho_{\th^0}^{-1}(3\e_0),
\eeaa
where $\rho_{\th^0}$ is  an invertible modulus of continuity of the path $\o^0$, and $\rho_{\th^0}^{-1}$ is the inverse function. Further, take a small neighborhood $O_{\e_2}$ of $\th^0$, where 
$$\e_2 ~<~ \e_1\we \rho_{\th^0}^{-1}(\e_1).$$
We next introduce two stopping times:
\beaa 
\ch_0(\o):= \inf\{t\ge 0: \th\in O_{\e_2}\}
\q\mbox{and}\q
\ch_1(\o) := \inf\{t\ge \ch_0(\o): |\o_t - \o^0_{t^0}|\ge 3\e_1 \}\we (t^0+3\e_1),
\eeaa
together with the set:
\beaa
Q:=\big\{\th\in \Th: \ch_0(\o)\le t\le \ch_1(\o)\big\}.
\eeaa
We claim and will prove in Step 5 that
\beaa
O_{\e_2}\subset Q\subset O_{9\e_0}.
\eeaa
In particular, we have $\ch_0(\o^0)<t^0<\ch_1(\o^0)$, and thus $\ch_1^{\th^0} -t^0 = \ch$. Since $(\a,\b)\in \ol\cJ_L u_*(\th)$, we have
\beaa
u_*(\th^0) < \ul\cE_L \big[(u_*)^{\th^0}_{\ch_1^{\th^0}-t^0} -(\a-\d) (\ch_1^{\th^0} -t^0)- \b B_{\ch_1^{\th^0}-t^0}\big].
\eeaa
We next define the inf-convolution of $u_*$:
\be\label{defn: ubar n}
\ul u^n(\th):=\inf_{\th'\in \Th} \{u_*(\th') + nd(\th',\th)\}\q\mbox{for all}~~\th\in\Th.
\ee
Notice that $\ul u^n$ is Lipschitz continuous. Since $u_*\in \mbox{LSC}_b(\Th)$, it is easy to show that $\ul u^n \uparrow u_*$. Thus, by \eqref{defn: ubar n}, we deduce that for $n$ sufficiently large
\beaa
u_*(\th^0) < \ul\cE_L \big[(\ul u^n)^{\th^0}_{\ch_1^{\th^0}-t^0} -(\a-\d) (\ch_1^{\th^0} -t^0)- \b B_{\ch_1^{\th^0}-t^0}\big].
\eeaa
By defining $\f(\th) := \ul\cE_L\Big[(\ul u^n)^{\th}_{\ch^\th_1-t}-(\a-\d)(\ch^{\th}_1-t) -\b B_{\ch^{\th}_1-t}\Big]$ for all $\th\in \Th$, we have
\be\label{eq:obj contradiction}
\f(\th^0) ~>~ u_*(\th^0).
\ee
We finally define
\beaa
U~:=~ (\f\vee u)1_Q + u1_{Q^c}.
\eeaa

\ms
\no {\bf 2.}\q In this step, we show that $\f$ is viscosity subsolution of the equation:
\be\label{eq: phi equation}
-\cL w - F(\cd, \f\vee u, \pa_\o w)\le 0,
\q\mbox{on}~~\{\th: t<\ch_1(\o)\}.
\ee
It follows from Proposition \ref{prop: especial sol} that for all $(\a',\b')\in \ul\cJ_L\f(\th)$, it holds that
\beaa
-\a'+\a -\d +L|\b-\b'|\le 0.
\eeaa
Further, by \eqref{eq: u delta neighbor} we obtain that
\beaa
-\a'-F\big(\th,(\f\vee u)(\th),\b'\big)\le -\a' -F(\th,u_*(\th),\b)+L|\b-\b'|\le 0.
\eeaa
So the desired result follows.

\ms
\no {\bf 3.}\q In this step, we prove that $U$ is a viscosity subsolution of Equation \eqref{PPDE}. First, for $\th\in Q^o:=\{\th:\ch_0(\o)\le t<\ch_1(\o)\}$, it is clear that both $\f$ and $u$ are viscosity subsolutions of Equation \eqref{eq: phi equation}. Then take any $(\a',\b')\in\ul\cJ_L U(\th)$, i.e.
\beaa
U(\th) = \max_{\t\in\cT_{\ch'}}\ol\cE_L[U^{\th}_\t -\a'\t-\b'B_\t]\q\mbox{for some}~~\ch'\in\cT^+.
\eeaa
If $u(\th)\le \f(\th)$, then it follows that
\beaa
\f(\th) \ge \ol\cE_L[U^{\th'}_\t -\a'\t-\b'B_\t]\ge \ol\cE_L[\f^{\th'}_\t -\a'\t-\b'B_\t]\q\mbox{for all}~~\t\in\cT_{\ch'}.
\eeaa
Thus $(\a',\b')\in \ul\cJ_L \f(\th)$. Otherwise, if $u(\th)>\f(\th)$, we may similarly get $(\a',\b')\in \ul\cJ_L u(\th)$. In both cases, it follows that
\beaa
-\a'-F\big(\th,(\f\vee u)(\th),\b'\big)\le 0.
\eeaa
So we have proved that $U$ is a viscosity subsolution of Equation \eqref{PPDE} on $Q^o$.

On the other hand, for $\th\in (Q^o)^c$, we have $U(\th)= u(\th)$, because whenever $t=\ch_1(\o)$ we have $\f(\th)=\ul u^n(\th)\le u_*(\th)\le u(\th)$. Then it becomes trivial to verify that $U$ is a viscosity subsolution of Equation \eqref{PPDE} on $(Q^o)^c$.

\ms
\no {\bf 4.}\q Our objective is to construct a viscosity subsolution in $\mbox{USC}_b(\Th)$. Since we did not prove $Q$ is closed, we do not know whether $U\in \mbox{USC}_b(\Th)$ itself. We next prove that the USC envelop $U^*$ is still a viscosity subsolution of Equation \eqref{PPDE}. Take any $(\a',\b')\in \ul\cJ_L U^*(\th)$. By the definition of $U^*$, there exists a sequence $\{\th^n\}\subset\Th$ such that
\beaa
d(\th^n,\th)\rightarrow 0,
\q\mbox{and}\q
\lim_{n\rightarrow\infty}U(\th^n)= U^*(\th).
\eeaa
Further, by (ii) of Proposition \ref{prop: especial sol}, $U$ is pathwise u.s.c. Consequently, we can apply Lemma \ref{lem: stability} and obtain that for any $n\in\dbN$ and $\e'>0$, there exits $N_n\ge n$ and $\hat\th^n$ such that
\beaa
d(\th^{N_n},\hat\th^n)\le \frac{1}{n}\q\mbox{and}\q
(\a'+\e',\b')\in {\rm cl}\Big(\ul\cJ_L U(\hat\th^n)\Big).
\eeaa
Since $U$ is a viscosity subsolution of Equation \eqref{PPDE} and $F$ is non-decreasing in $y$, we have
\beaa
-\a'-\e'-F(\hat\th^n,U^*(\hat\th^n),\b')
~\le ~ -\a'-\e'-F(\hat\th^n,U(\hat\th^n),\b')
~\le ~ 0.
\eeaa
Letting $n\rightarrow\infty$ and $\e'\rightarrow 0$, we get
\beaa
-\a'-F(\th,U^*(\th),\b')~\le ~ 0.
\eeaa

Then it is clear that $U^*\in \cD$, so $U^*\le u$ on $\Th$. On the other hand, there exists a sequence $\{\th^n\}\subset O_{\e_2}$ such that $u_*(\th^0)=\lim_{n\rightarrow\infty}u(\th^n)$. Also, by Proposition \ref{prop: especial sol}, $\f$ is continuous on $Q\supset O_{\e_2}$. Then by \eqref{eq:obj contradiction} we have
\beaa
\liminf_{n\rightarrow\infty}(U^*-u)(\th^n)
\ge  \lim_{n\rightarrow\infty}(\f-u)(\th^n)
= \f(\th^0) - u_*(\th^0)>0.
\eeaa
Therefore, there is $\th^n$ such that $U^*(\th^n)>u(\th^n)$. That is a contradiction to $U^*\in \cD$.

\ms
\no {\bf 5.}\q We finally complete the proof of $O_{\e_2}\subset Q\subset O_{9\e_0}$. First, for all $\th\in O_{\e_2}$, it is clear that $\ch_0(\o)\le t$. We denote $t_0:=\ch_0(\o)$ and then consider $s\in [t_0,t]$. Since $|t_0-t^0|\le \e_2$ and $|t-t^0|\le \e_2$, we have $|s-t^0|\le \e_2$. Further, since $\th\in O_{\e_2}$, we have
\beaa
|\o_s-\o^0_{t^0\we s}| ~\le ~ d\big(\th, \th^0\big) ~\le ~ \e_2 ~\le ~ \e_1,
\eeaa
and
\beaa
|\o_s-\o^0_{t^0}|~\le ~ |\o_s-\o^0_{t^0\we s}|+|\o^0_{t^0\we s}-\o^0_{t^0}|
~\le ~ \e_1+\rho_{\th^0}(t^0-t^0\we s)~< ~2\e_1.
\eeaa
It follows that $\ch_1(\o)\ge t$, and thus $\th\in Q$.

Next, take any $\th\in Q$. Still denote $t_0:= \ch_0(\o)$. For $s\le t_0$, since $(t_0,\o)\in O_{\e_2}$, it is clear that
\beaa
|\o_s-\o^0_{t^0\we s}| ~\le ~ d\big((t_0,\o),\th^0\big) ~\le ~ \e_2.
\eeaa
On the other hand, for $s\in [t_0,t]$, since $s\le t\le \ch_1(\o)$, it holds
\beaa
|t^0-s|\le 3\e_1<3\e_0
\q\mbox{and}\q
|\o_s-\o^0_{t^0\we s}|\le |\o_s - \o^0_{t^0}| + |\o^0_{t^0\we s}-\o^0_{t^0}| 
\le 3\e_1 + \rho_{\th^0}(t^0-t^0\we s) <6\e_0.
\eeaa
It follows that $d(\th,\th^0)<9\e_0$, and thus $\th\in O_{9\e_0}$.
\qed

\section{Comparison result}\label{sec: comparison}

\subsection{Regularization}\label{sec: regularization}

For a viscosity subsolution $u\in \mbox{USC}_b(\Th)$ and a viscosity supersolution $v\in \mbox{LSC}_b(\Th)$, we define $M:=\sup_{\th\in\Th}\big(|u(\th)|\vee|v(\th)|\big)$, and
\be\label{defn: un}
u^n(\th)~:=~ \sup_{\th'\in\Th} \big(u(\th') - n\overleftarrow{d}(\th,\th')\big),\q v^n(\th)~:=~ \inf_{\th'\in\Th} \big(v(\th') + n\overleftarrow{d}(\th,\th')\big),
\ee

\begin{lem}\label{lem: u^n limit}
For each $n$, $u^n$ is bounded, Lipschitz continuous in $\overleftarrow{d}(\cd,\cd)$, and  continuous in $d(\cd,\cd)$. Moreover, $u^n$ is decreasing in $n$ and $\lim_{n\rightarrow\infty} u^n(\th)=u(\th)$, for all $\th\in\Th$. The similar result holds true for $v^n$.
\end{lem}
\begin{proof}
Clearly, $u^n$ is bounded and Lipschitz continuous in $\overleftarrow{d}(\cd,\cd)$, for each $n$. By Lemma \ref{lem: difference d,d}, $u^n$ is also continuous in $d(\cd,\cd)$. Also, it is clear that $u^n$ is decreasing in $n$ and $u^n \ge u$ for each $n$. Define $u^\infty:=\lim_{n\rightarrow\infty}u^n$. Then, $u^\infty\ge u$. On the other hand, since $u$ is bounded, we have
\beaa
u^n(\th)~:=~ \sup_{\overleftarrow{d}(\th',\th)\le\frac{2M}{n}} \big(u(\th') - n\overleftarrow{d}(\th,\th')\big)
\eeaa
In particular, there exists $\th^n$ such that
\beaa
\overleftarrow{d}(\th^n,\th)\le \frac{2M}{n}
\q\mbox{and}\q
u^n(\th)\le u(\th^n)+\frac{1}{n}.
\eeaa
Therefore, $u^\infty(\th)\le \limsup_{n\rightarrow\infty} u(\th^n)$. Since $u\in \mbox{USC}_b(\Th)$, it follows that
\beaa
u^\infty(\th)\le \limsup_{n\rightarrow\infty} u(\th^n)\le u(\th).
\eeaa
\end{proof}

\subsection{Generator $F(\th,y,z)$ independent of $y$}\label{sec: independ y}

In this subsection we suppose that there is no dependence on $y$ in the generator $F(\th,y,z)$.  Let $u\in\mbox{USC}_b(\Th)$ be a viscosity subsolution of the path dependent PDE with the generator $F(\th,y,z) = F_0(\th,z)$, and  $v\in\mbox{LSC}_b(\Th)$ be a viscosity supersolution of the path dependent PDE with the generator $F(\th,y,z) = F_0(\th,z)+\d(\th)$. We suppose that Assumption \ref{assum: F} holds true for both generators $F_0$ and $F_0+\d$. In particular, we denote $\rho^0:=\rho^{F_0} \vee \rho^{F_0+\d}$.

\begin{prop}\label{prop: un}
For each $n$, $u^n$ is a viscosity subsolution of the following path dependent PDE:
\be\label{eq: u^n}
-\cL u^n(\th)-F_0(\th,\pa_\o u^n(\th)) - \rho^0 (\th,\e_n(\th))\le 0,
\ee
where $\e_n(\th):=\frac{2M+1}{n}+\bar\rho\big(\th,\frac{2M}{n}\big)$. Similarly, $v^n$ is a viscosity supersolution of:
\be\label{eq: v^n}
-\cL v^n(\th)-F_0(\th,\pa_\o v^n(\th)) +\d(\th) + \rho^0 (\th,\e_n(\th))\ge 0.
\ee
\end{prop}
\begin{proof}
We only prove the result for $u^n$. Let $(\a,\b)\in \ul\cJ_L u^n(\th)$, i.e.
\beaa
u^n(\th) = \max_{\t\in\cT_\ch}\ol\cE_L\big[(u^n)^\th_\t-\a\t-\b B_\t\big],\q\mbox{for some}~~\ch\in\cT^+.
\eeaa
Without loss of generality, we may assume that $\ch(\o')\le \ch^n(\o'):=\inf\{t':|t'|+\|\o'_{t'\we\cd}\|\ge \frac{1}{n}\}$ for all $\o'\in\O$.
For any $\e>0$, we have
\be\label{u^n eps}
u^n(\th)-c > \ol\cE_L\big[(u^n)^\th_\ch-(\a+\e)\ch-\b B_\ch\big],
\q\mbox{for some}~~c>0.
\ee
By the definition of $u^n$ and $|u|\le M$, there exists $\th^n=(t^n,\o^n)\in\Th$ such that
\be\label{c optimal th}
\overleftarrow{d}(\th,\th^n)\le \frac{2M}{n}
\q\mbox{and}\q
u^n(\th)-c\le u(\th^n)-n\overleftarrow{d}(\th,\th^n).
\ee
Further, since $(u^n)^\th_\ch\ge u^{\th^n}_\ch - n\overleftarrow{d}\big((t+\ch,\o\otimes_t B),(t^n+\ch,\o^n\otimes_{t^n}B)\big) = u^{\th^n}_\ch - n\overleftarrow{d}(\th,\th^n)$, it follows from \eqref{u^n eps} and \eqref{c optimal th} that
\beaa
u(\th^n) >
\ol\cE_L\big[u^{\th^n}_\ch -(\a+\e)\ch-\b B_\ch\big]
\eeaa
By Lemma \ref{lem: fundamental} and $\ch\le \ch^n$, we may find $\bar\th^n\in\Th$ such that
\beaa
\big(\a+\e,\b\big)\in\ul\cJ_L u(\bar\th^n)
\q\mbox{and}\q
d(\th^n,\bar\th^n)\le \frac{1}{n}.
\eeaa
Since $u$ is a viscosity subsolution, we have
\be\label{a b simple}
-(\a+\e)-F_0(\bar\th^n,\b)\le 0.
\ee
Further, by Assumption \ref{assum: F}, we obtain that
\be\label{estim: F th th^n}
| F_0(\bar\th^n,\b)-F_0(\th,\b)|
\le 
\rho^0\big(\th,d(\th,\bar\th^n)\big)
\le
\rho^0\big(\th,d(\th,\th^n)+d(\th^n,\bar\th^n)\big)\le \rho^0\Big(\th,d(\th,\th^n)+\frac{1}{n}\Big).
\ee
By Lemma \ref{lem: difference d,d} and \eqref{c optimal th}, we have
\beaa
d(\th,\th^n)\le \overleftarrow{d}(\th,\th^n)+\bar\rho(\th,|t-t^n|)\le \frac{2M}{n}+\bar\rho\Big(\th,\frac{2M}{n}\Big).
\eeaa
It follows from \eqref{a b simple} and \eqref{estim: F th th^n} that
\beaa
-(\a+\e)-F_0(\th,\b)-\rho^0 (\th,\e_n(\th))\le 0.
\eeaa
Finally, by letting $\e\rightarrow 0$, we show that $u^n$ is a viscosity subsolution of the path dependent PDE \eqref{eq: u^n}.
\end{proof}

In Proposition 4.17 of \cite{RTZ} the authors proved that if $u,v$ are viscosity subsolution and supersolution of the same path dependent PDE, then $u-v$ is a viscosity subsolution of the equation $-\cL w - L|w| -L|\pa_\o w|=0 $. Here, although $u^n, v^n$ are corresponding to two different equations, one may follow the same argument as in \cite{RTZ} and prove that:

\begin{prop}\label{prop: wn}
Denote $w^n:=u^n-v^n$. Then $w^n\in\mbox{\rm USC}_b$ is a viscosity subsolution of the path dependent PDE:
\be\label{eq: wn}
-\cL w^n(\th) - L|\pa_\o w^n(\th)|\le 2\rho^0(\th,\e_n(\th))+\d(\th).
\ee
\end{prop}

\begin{prop}\label{prop: comparison special}
Denote $w:=u-v$. Then $w=\lim_{n\rightarrow\infty}w^n$ and is a viscosity subsolution of
\be\label{eq: w}
-\cL w(\th) - L|\pa_\o w(\th)|~\le~ \d(\th).
\ee
\end{prop}

\begin{proof}
By Lemma \ref{lem: u^n limit}, we have $w=\lim_{n\rightarrow\infty}w^n$. Suppose $(\a,\b)\in\ul\cJ_L w(\th)$. Then by Lemma \ref{lem: stability}, for any $n$ and $\e>0$, there exists $N_n\ge n$ and $\hat\th^n$ such that
\beaa
d(\hat\th^n,\th)\le \frac{1}{n}\q\mbox{and}\q(\a+\e,\b)\in\ul\cJ_L w^{N_n}(\hat\th^n)
\eeaa
By Proposition \ref{prop: wn}, $w^n$ is a viscosity subsolution of equation \eqref{eq: wn}. Therefore,
\beaa
-(\a+\e)-L|\b| \le 2\rho^0(\hat\th^n,\e_n(\hat\th^n))+\d(\hat\th^n).
\eeaa
Let $n\rightarrow\infty$ and then $\e\rightarrow 0$. It follows that $-\a-L|\b|\le 0$. So we verified that $w$ is a viscosity subsolution of equation \eqref{eq: w}.
\end{proof}

\subsection{Maximum principle}

In this section, we study the equation corresponding to the Pucci's extremal operator:
\be\label{eq: extrema}
-\cL u - L u^+ - L|\pa_\o u|=0.
\ee

\begin{prop}[Maximum principle]\label{prop: extrema}
Let $u\in \mbox{\rm USC}_b(\Th)$ be a viscosity subsolution of Equation \eqref{eq: extrema}, and suppose that $u_T\le 0$. Then, we have $u\le 0$ on $\Th$. 
\end{prop}
In preparation of the proof of Proposition \ref{prop: extrema}, we need some observations. Recall the sup-convolution defined in \eqref{defn: un}. Since $u\le u^m$, we clearly have:
\begin{lem}
If $u\in \mbox{\rm USC}_b(\Th)$ is a viscosity subsolution of Equation \eqref{eq: extrema}, then $u$ is also a viscosity subsolution of:
\be\label{eq: extrema um}
-\cL u - L (u^m)^+ - L|\pa_\o u|\le 0.
\ee
\end{lem}

For Equation \eqref{eq: extrema um}, the generator is:
\beaa
F^m(\th,z) = L(u^m(\th))^+ - L|z|.
\eeaa 
Further, we may estimate:
\beaa
|F^m(\th,z)-F^m(\th',z)| 
\le  L \big(u^m(\th)-u^m(\th')\big)^+
\le  Lm \overleftarrow{d}(\th,\th')\le  Lm\Big(d(\th,\th')+\bar\rho\big(\th,d(\th,\th')\big)\Big)
.
\eeaa 
Therefore, generator $F^m$ satisfies Assumption \ref{assum: F} and is among the generators independent of $y$ discussed in the previous section.

\ms
\no {\bf Proof of Proposition \ref{prop: extrema}}\q
By using the same argument as in the proof of Proposition \ref{prop: un}, we can prove that $u^n$ is a viscosity subsolution of
\beaa
-\cL u^n(\th)-L(u^m(\th))^+ - L|\pa_\o u^n(\th)| - \rho^{n,m}(\th)\le 0,
\eeaa
where  $\rho^{n,m}(\th):=Cm\Big(\frac{1}{n}+\bar\rho\big(\th,\frac{C}{n}\big)\Big)$ and $C$ is a sufficiently large constant. Clearly, $u^n$ is also a viscosity subsolution of:
\be\label{eq: un with um}
-\cL w(\th)-L(w(\th))^+ - L|\pa_\o w(\th)| \le \rho^{n,m}(\th)+L\big(u^m(\th)-u^n(\th)\big)^+.
\ee
Now we introduce a function $v^{n,m}$:
\beaa
v^{n,m}(\th):=\ol\cE_L\Big[\int_0^{T-t} e^{Ls}\Big((\rho^{n,m})^\th_s+L\big((u^m)^\th_s-(u^n)^\th_s\big)^+\Big)ds + e^{L(T-t)}\big((u^n)^\th_{T-t}\big)^+\Big].
\eeaa
As a value function of a stochastic optimal control problem, one may easily prove that $v^{n,m}$ is viscosity supersolution of Equation \eqref{eq: un with um}. Further it is clear that $v^{n,m}\in C(\Th)$ and $v^{n,m}_T=(u^n_T)^+$. Then by Theorem \ref{thm: comparison continuous}, we obtain that $u^n\le v^{n,m}$ on $\Th$. Now let $n\rightarrow \infty$, we have
\beaa
u(\th) \le \ol\cE_L \Big[\int_0^{T-t} e^{Ls}L\big((u^m)^\th_s-u^\th_s\big)^+ds \Big] \q\mbox{for all}~~\th\in\Th,
\eeaa
where we used the fact $u_T\le 0$. Finally, let $m\rightarrow\infty$, we get $u\le 0$ on $\Th$.
\qed

\subsection{Comparison result for general generators}

In this section we are going to prove the comparison result for equations in the general form \eqref{PPDE} under Assumption \ref{assum: F}. Similar to Proposition 3.14 in \cite{ETZ1} which provides a change of variable for continuous viscosity solutions, we show the following result on a change of variable for semi-continuous viscosity solutions.

\begin{lem}
Let $u\in \mbox{\rm USC}_b(\Th)$ be a viscosity subsolution of Equation \eqref{PPDE}. Define $\tilde u_t(\o):=e^{-Lt}u_t(\o)$. Then $\tilde u\in \mbox{\rm USC}_b(\Th)$ is a viscosity subsolution of the equation:
\beaa
-\cL \tilde u(\th)-L\tilde u(\th) - e^{-Lt} F\big(\th, e^{Lt} \tilde u(\th),e^{Lt}\pa_\o \tilde u(\th)\big)=0.
\eeaa
The similar result holds for viscosity supersolutions.
\end{lem}
\begin{proof}
Without loss of generality, we only verify the viscosity subsolution property at $0$. Let $(\a,\b)\in \ul\cJ_L \tilde u(0)$, i.e.
\beaa
\tilde u_0 = \max_{\t\in\cT_\ch}\ol\cE_L[\tilde u_\t-\a\t-\b B_\t]\q\mbox{for some}~~\ch\in\cT^+.
\eeaa
It means that
\be\label{eq: change}
u_0 = \max_{\t\in\cT_\ch}\ol\cE_L[e^{-L\t}u_\t - \a\t-\b B_\t].
\ee
Since we have
\beaa
\lim_{t\rightarrow 0}\frac{e^{-Lt}-1}{t}=-L
\q\mbox{and}\q
\limsup_{t\rightarrow 0} u_t \le u_0,
\eeaa
for $\e>0$ we may assume that
\beaa
e^{-Lt}-1+Lt \ge -\e t
\q\mbox{and}\q
u_t \le u_0+\e, \q\mbox{for all}~~t\le \ch.
\eeaa
From \eqref{eq: change}, we obtain that for all $\t\in\cT_\ch$
\beaa
u_0 &\ge & \ol\cE_L\big[(e^{-L\t}-1+L\t)u_\t+u_\t-L\t u_\t-\a\t-\b B_\t\big]\\
&\ge & \ol\cE_L\big[\e C\t + u_\t -L(u_0+\e)\t -\a\t -\b B_\t \big].
\eeaa
This implies that $\big(\a+Lu_0 +(L-C)\e,\b\big)\in \ul\cJ_L u(0)$. Thus
\beaa
-\a-Lu_0 -(L-C)\e-F(0,u_0,\b)\le 0.
\eeaa
By letting $\e\rightarrow 0$, we obtain the desired result.
\end{proof}

\begin{rem}{\rm
For continuous viscosity solutions, the previous result holds true for the change of variables of the form of $\tilde u_t(\o):=e^{\l t}u_t(\o)$ for all $\l \in \dbR$. However, as showed in the previous lemma, the same result only holds true for $\l\le 0$ in the context of semi-continuous viscosity solutions.
}
\end{rem}

Due to the previous lemma, without loss of generality we may assume that the generator $F:(\th,y,z)\mapsto\dbR$ is non-decreasing in $y$. 

\ms
\no {\bf Proof of Theorem \ref{thm: comparison}}\q
 Since $u^n\ge u$,  $u$ is a viscosity subsolution of the equation:
\beaa
-\cL u(\th) - F\big(\th, u^n(\th),\pa_\o u(\th)\big)\le 0.
\eeaa
Similarly,  $v$ is a viscosity supersolution of the equation:
\beaa
-\cL v(\th) - F\big(\th, u^n(\th),\pa_\o u(\th)\big) + L\big(u^n(\th)-v^n(\th)\big)^+
\ge -\cL v(\th) - F\big(\th, v^n(\th),\pa_\o u(\th)\big)\ge 0.
\eeaa
Consider the generator $F^n(\th,z):=F(\th,u^n(\th),z)$, and observe that
\beaa
|F^n(\th,z)-F^n(\th',z)| &=& |F(\th,u^n(\th),z)-F^n(\th',u^n(\th'),z)| \\
&\le & Ln\overleftarrow{d}(\th,\th')+\rho^F\big(\th,d(\th,\th'),u^n(\th)\big)\\
&\le & Ln\Big(d(\th,\th')+\bar\rho\big(\th,d(\th,\th')\big)\Big)+\rho^F\big(\th,d(\th,\th'),u^n(\th)\big)\\
&=: & \rho^{F^n}\big(\th,d(\th,\th')\big).
\eeaa
Therefore, the generator $F^n$ is of the type discussed in the previous section. So by setting $\d(\th):=L\big(u^n(\th)-v^n(\th)\big)^+$, we obtain from Proposition \ref{prop: comparison special} that $w:=u-v$ is a viscosity subsolution of the equation:
\beaa
-\cL w(\th) - L|\pa_\o w(\th)|\le L\big(u^n(\th)-v^n(\th)\big)^+,
\q\mbox{for each}~~n.
\eeaa
Further, by letting $n\rightarrow\infty$, we have that $w$ is a viscosity subsolution of Equation \eqref{eq: extrema}. Finally, by the maximum principle (Proposition \ref{prop: extrema}) we conclude that $w=u-v\le 0$ on $\Th$.
\qed

\section{Optimal stopping for semicontinuous barriers}\label{sec: Optimal stopping}

This section is devoted to the proof of Theorem \ref{thm: OS}. Denote
$$\ol\cE_L[\cd|\cF_t] ~:=~ \esup_{\dbP\in\cP}\dbE^\dbP[\cd|\cF_t].$$
By standard argument, we may prove:

\begin{lem}\label{lem: tower}
For any $\ol\cE_L$-uniformly integrable r.v. $X$, it holds that
$$\ol\cE_L[X|\cF_t] ~=~ \ol\cE_L\big[\ol\cE_L[X|\cF_s]\big|\cF_t\big],~~\dbP_0\mbox{-a.s., for all}~t\le s.$$
\end{lem}

We consider the optimal stopping problem:
\beaa
Y_0 &:= & \sup_{\t\in\cT^*}\ol\cE_L\big[X_\t\big],
\eeaa
where $X$ is a process u.s.c. in $t$. Define the dynamic version of the optimal stopping problem:
\beaa
Y_t~:= ~ \esup_{\t\in\cT_*^t} \ol\cE_L[X_\t|\cF_t]~:= ~\esup_{\t\in\cT_*^t,\dbP\in\cP_L}\dbE^\dbP\big[X_\t\big|\cF_t\big],
\eeaa
where $\cT_*^t$ is the set of all the stopping times in $\cT_*$ larger than $t$.

\subsection{Doob-Meyer decomposition}\label{sec:DM}

In most of the existing literature, authors only discuss the Doob-Meyer decomposition for RCLL supermartingale in class D. However, in our case, we need the decomposition under some weaker conditions. We find that the argument in Beiglb\"ock, Schachermayer and Veliyev \cite{BSV} can deduce a variation of the classical Doob-Meyer decomposition which serves well our purpose. In this subsection, we will quickly review their result and prove the decomposition theorem (Proposition \ref{prop: DM on D}).

Let $Y$ be a $\dbP$-supermartingale for some probability measure $\dbP$. Denote 
$$\cD_n:=\Big\{\frac{j}{2^n}:j\in\dbN,\frac{j}{2^n}\le T\Big\}\q \mbox{and}\q
\cD := \cup_n \cD_n.$$
For each $n$, we have the discrete time Doob-Meyer decomposition:
\beaa
Y_t = Y_0 + M^n_t -A^n_t, ~\mbox{for all}~t\in \cD^n,~\dbP\mbox{-a.s.}
\eeaa
According to Lemma 2.1 and 2.2 in \cite{BSV}, we have:
\begin{lem}\label{lem:BSV}
{\rm (i).}\q  Let $\{f_n\}_{n\ge 1}$ be a $\dbP$-uniformly integrable sequence of functions. Then there exists functions $g_n\in \mbox{conv}(f_n,f_{n+1},\cds)$ such that $\{g_n\}_{n\ge 1}$ converges in $\|\cd\|_{L^1(\dbP)}$.

\no {\rm (ii).}\q Assume that $\{Y_\t\}_{\t\in\cT_\cD}$ is $\dbP$-uniformly integrable, where $\cT_\cD$ is the set of stopping times in $\cT_*$ taking values in $\cD$. Then the sequence $\{M^n_T\}_{n\ge 1}$ is $\dbP$-uniformly integrable.
\end{lem}

Then following the same argument as in \cite{BSV}, we obtain the following result.
\begin{prop}\label{prop: DM on D}
Let $Y$ be $\dbP$-supermartingale such that $\{Y_\t\}_{\t\in\cT_\cD}$ is $\dbP$-uniformly integrable. Then there exists a martingale $M$ and an adapted non-decreasing process $A$ both starting from $0$ such that
\be\label{DM on D}
Y_t=Y_0+M_t-A_t,~~\mbox{for all}~~t\in\cD,~\dbP\mbox{-a.s.}
\ee
\end{prop}
\begin{proof}
For each $n$, extend $M^n$ to a cadlag martingle on $[0,T]$ by setting $M^n_t:=\dbE^\dbP[M^n_T|\cF_t]$. By Lemma \ref{lem:BSV}, there exist $M\in L^1(\dbP)$ and for each $n$ convex weights $\l^n_n,\cds, \l^n_{N_n}$ such that with
\beaa
\cM^n := \l^n_n M^n+\cds+\l^n_{N_n} M^{N_n}
\eeaa
we have $\cM^n_1\rightarrow M$ in $L^1(\dbP)$. Then, by Jensen's inequality, $\cM^n_t\rightarrow M_t:=\dbE^\dbP[M|\cF_t]$ for all $t\in [0,T]$. For each $n$ we extend $A^n$ to $[0,T]$ by $A^n:=\sum_{t\in\cD_n}A^n_t 1_{(t-\frac{1}{2^n},t]}$ and set:
\beaa
\cA^n := \l^n_n A^n+\cds+\l^n_{N_n} A^{N_n}.
\eeaa
Then the process $\bar A:=M+Y_0-Y$ satisfies for every $t\in\cD$
\beaa
\cA^n_t=\cM^n_t +Y_0-Y_t\longrightarrow M_t+Y_0-Y_t=\bar A_t~~\mbox{in}~L^1(\dbP).
\eeaa
Therefore, $\bar A$ is a.s. non-decreasing on $\cD$, $\dbP$-a.s. Finally, the process $A_t:= \sup_{s\le t,s\in\cD}\bar A_s$ is non-decreasing on $[0,T]$, $\dbP$-a.s., and satisfies \eqref{DM on D}.
\end{proof}

\begin{rem}{\rm
In \cite{BSV}, by further assuming that $Y$ is cadlag and in class D, we may get the decomposition on $[0,T]$, and prove that process $A$ is previsible.
}
\end{rem}

\subsection{Skorokhod decomposition for lower semicontinuous functions}\label{sec:skorokhod}

\begin{lem}\label{lem:skorokhod}
Let $\l:[0,T]\rightarrow \dbR$ be lower semicontinuous (l.s.c.) with $\l_0=0$, and define
\beaa
\k_t := \max_{s\le t} \l^-_s = - \min_{s\le t} \l_s
\q\mbox{and}\q
\eta_t := \l_t + \max_{s\le t} \l^-_s.
\eeaa
Then,

\no {\rm (i).}\q $\eta$ is non-negative and $\k$ is non-decreasing, such that
\beaa
\eta_0=\k_0=0,\q \l_t = \eta_t - \k_t ~~ \mbox{for all}~~t\in [0,T].
\eeaa

\no {\rm (ii).}\q $\eta$ is l.s.c., $\k$ is right continuous, and it holds that
\beaa
\int_0^T 1_{\{\eta_t \neq 0\}} d\k_t=0.
\eeaa

\no {\rm (iii).}\q for all other non-negative function $\eta'$ and  non-decreasing function $\k'$ satisfying (i), it holds
\beaa
\k_t\le \k'_t \q\mbox{for all}\q t\in [0,T].
\eeaa
\end{lem}
\begin{proof}
(i) is trivial. We only prove (ii) and (iii).

\no (ii).\q First, we claim that
\be\label{min lambda}
\min_{r\le t}\l_r = \liminf_{s\rightarrow t}\min_{r\le s}\l_r.
\ee
Since $\l_t = \liminf_{s\rightarrow t}\l_s$, it is clear that $\min_{r\le t}\l_r \ge \liminf_{s\rightarrow t}\min_{r\le s}\l_r$. On the other hand, we  have
$$\min_{r\le t-\e}\l_r \le \liminf_{s\rightarrow t}\min_{r\le s}\l_r,\q\mbox{for all}~~\e>0.$$
It implies that $\inf_{r < t}\l_r \le \liminf_{s\rightarrow t}\min_{r\le s}\l_r$. Again by $\l_t = \liminf_{s\rightarrow t}\l_s$, we obtain that $\inf_{r < t}\l_r \ge \min_{r\le t}\l_r$. So we proved \eqref{min lambda}. Consequently, by the definition of $\k$, we have $\k_t = \limsup_{s\rightarrow t}\k_s$. Taking into account that $\k$ is non-decreasing, we obtain that $\k_t = \lim_{s\downarrow t}\k_s$.

 For any $\e>0$, take $t\in \big\{s:\eta_s> \e\big\}$, i.e.
\beaa
\l_t + a >\e,~\mbox{where}~a:=\k_t.
\eeaa
Since $\l$ is l.s.c., the set $\big\{s:\l_s>-a+\e\big\}$ is open. Thus, there is an open neighborhood $O_t$ of $t$ on which $\l>-a +\e$. We claim that 
\be\label{O_t}
\l > -\k +\e \q \mbox{on}\q O_t.
\ee
Suppose to the contrary, i.e. there exists $\bar t\in O_t$ such that $\l_{\bar t}\le -\k_{\bar t}+\e$. If $\bar t\ge t$, then $\l_{\bar t}\le -\k_{\bar t}+\e\le -\k_t+\e = -a+\e$, which is a contradiction. Otherwise, if $\bar t<t$, since $-a+\e<\l_{\bar t}\le -\k_{\bar t}+\e$, we obtain that $\k_{\bar t}< a$. However, since $\k_t=a$, there exists $\hat t\in [\bar t, t]$ such that $\l_{\hat t} = -a$, which is also a contradiction. So we proved \eqref{O_t}. It follows that $\big\{s:\eta_s> \e\big\}$ is open for all $\e>0$, and thus $\eta$ is l.s.c. 

On the other hand, since $\big\{s:\eta_s> \e\big\}$ is open, it can be written as the union of a countable number of open intervals, i.e. $\big\{s:\eta_s> \e\big\} = \cup_n (s_n,t_n)$. Since $(s_n,t_n)\subset \big\{s:\eta_s> \e\big\}$, we clearly have $\k_{t_n-}-\k_{s_n}=0$. Further, we have 
$$\int_0^T 1_{\{\eta_s >\e \}}d\k_s=\sum_n (\k_{t_n-}-\k_{s_n})=0.$$
Finally, it follows from the monotone convergence theorem that $\int_0^T 1_{\{\eta_s > 0 \}}d\k_s =0$.

\ms
\no (iii).\q Assume to the contrary, i.e. let $t\in (0,T]$ such that $\k_t>\k'_t$. Take $s^*:=\sup\{s\le t: \eta_s =0\}$. Since $\eta$ is non-negative and l.s.c., the set $\{\eta=0\}$ is closed, and therefore, $\eta_{s^*}=0$. Also, since $(s^*,t]\subset \{\eta>0\}$, we have $\k_t-\k_{s^*}=0$. Then,
\beaa
\eta'_{s^*} = \eta_{s^*} - \k_{s^*} + \k'_{s^*}\le  \k'_t-\k_t < 0,
\eeaa
contradiction.
\end{proof}

\subsection{Optimal stopping for upper semicontinuous barriers}

\begin{lem}\label{lem:UI}
$Y$ is an $\dbF^*$-adapted $\ol\cE_L$-supermartingale. Moreover, $\{Y_\t\}_{\t\in \cT_\cD}$ is $\ol\cE_L$-uniformly integrable.
\end{lem}
\begin{proof}
By standard argument, one may prove the first part of the lemma. We are going to prove the second part, by showing that $\{Y^+_\t\}_{\t\in\cT_\cD}$ and $\{Y^+_\t\}_{\t\in\cT_\cD}$ are both $\ol\cE_L$-uniformly integrable.

\no {\bf 1.}\q By the definition of $Y$, it is clear that $Y_t \le \ol\cE_L[\sup_{s\in [0,T]}X_s|\cF_t]$. Further, by Jensen's inequality, it follows that $Y^+_t \le \ol\cE_L[\sup_{s\in [0,T]}X^+_s|\cF_t]$. Then for all $\t\in\cT_\cD$ we have
\beaa
Y^+_\t~ \le ~\ol\cE_L[\sup_{s\in [0,T]}X^+_s|\cF_\t],\q \dbP_0\mbox{-a.s.}
\eeaa
By  (ii) of the assumptions of Theorem \ref{thm: OS},  it is easy to prove that $\{Y^+_\t\}_{\t\in\cT_\cD}$ is $\ol\cE_L$-uniformly integrable.

\no {\bf 2.}\q Since $Y$ is a $\dbP$-supermartingale for all $\dbP\in\cP$, $Y^-$ is a $\dbP$-submartingale for all $\dbP\in\cP$. Consequently, we have
\beaa
Y^-_\t~ \le ~\ol\cE_L[Y^-_T|\cF_\t] ~ = ~\ol\cE_L[X^-_T|\cF_\t].
\eeaa
By (iii) of the assumptions of Theorem \ref{thm: OS}, one may easily prove that $\{Y^-_\t\}_{\t\in\cT_\cD}$ is $\ol\cE_L$-uniformly integrable.
\end{proof}

\begin{rem}{\rm
In the previous proof, it is crucial to consider the $\ol\cE_L$-uniform integrability of $\{Y_\t\}_{\t\in\cT_\cD}$ instead of $\{Y_\t\}_{\t\in\cT_*}$.
}
\end{rem}

\begin{lem}
$Y$ has a left continuous version.
\end{lem}

\begin{proof}
{\bf 1.}\q We first prove $\lim_{s\uparrow t}\ul\cE_L[Y_s-Y_t]=0$.
Since $Y$ is a supermartingale, it is sufficient to prove that 
\be\label{eq: EY}
\limsup_{s\uparrow t} \ul\cE_L[Y_s-Y_t]\le 0.
\ee
Since $Y\ge X$, $\dbP_0$-a.s., it follows from Lemma \ref{lem: tower} that
\beaa
\ul\cE_L[Y_s-Y_t] & = & \ul\cE_L\Big[\esup_{\t\in\cT_*^s}\ol\cE_L[X_\t|\cF_s]-\ol\cE_L[Y_t|\cF_s]\Big]
\\
&\le & \ul\cE_L\Big[\esup_{\t\in\cT_*^s} 
	\ol\cE_L\big[X_\t 1_{\{\t<t\}}+Y_t 1_{\{\t\ge t\}}\big|\cF_s\big]-\ol\cE_L[Y_t|\cF_s]\Big]
\\
&\le &	\ol\cE_L\Big[\esup_{\t\in\cT_*^s}\ol\cE_L[(X_\t-Y_t)^+|\cF_s]\Big]
\\
&\le & \ol\cE_L\Big[\ol\cE_L[(\ol X_s^t-Y_t)^+|\cF_s]\Big]
~ =~ \ol\cE_L[(\ol X_s^t-X_t)^+],
\eeaa
where $\ol X^t_s:=\sup_{s\le r\le t}X_r$. Since $X$ is u.s.c. in $t$, it holds that 
$\lim_{s\uparrow t}\ol X^t_s\le X_t$. Further, in view of (ii) and (iii) of the assumptions of Theorem \ref{thm: OS}, \eqref{eq: EY} follows from Lemma \ref{lem: dominate_conv}.

\ms
\no {\bf 2.}\q It follows from Lemma \ref{lem:UI} that $Y$ is a $\dbP_0$-supermartingale in the continuous filtration $\dbF^*$. By classical martingale theory, we know that for any $t\in[0,T)$,
$$Y_{t-}:=\lim_{s\upuparrows t,s\in\cD}Y_s\text{ exists }\dbP_0\text{-a.s.},$$
and that $\{Y_{t-}\}_t$ is left continuous and $Y_t=\dbE[Y_t|\cF^*_{t-}]\leq Y_{t-}$, $\dbP_0$-a.s. We next show that $Y_{t-}=Y_t$, $\dbP_0$-a.s. Suppose to the contrary that $\dbP_0[Y_t<Y_{t-}]>0$. Then, we have $\dbE^{\dbP_0}\big[\sqrt{Y_{t-}-Y_t}\big]>0$, implying that $\ul\cE_L[Y_{t-}-Y_t]>0$.  On the other hand, it follows from the result of Step 1 and Lemma \ref{lem:UI} that
\beaa
0 = \lim_{s\upuparrows t,s\in\cD}\ul\cE_L[Y_s-Y_t] =\ul\cE_L[Y_{t-}-Y_t]>0,
\eeaa
contradiction.
\end{proof}

Then following the discussion in Section \ref{sec:DM},  we can show that:

\begin{lem}\label{lem: DM}
For all $\dbP\in\cP$, there exists a $\dbP$-martingale $M^\dbP$ and a non-decreasing process $A^\dbP$ such that
\be\label{DM decom}
Y_t = Y_0+M^\dbP_t-A^\dbP_t,~\mbox{for all}~t\in [0,T],~\dbP_0\mbox{-a.s.}
\ee
In particular, there exists $Z$ such that $M^{\dbP_0}=\int_0^\cd Z_tdB_t$, $\dbP_0$-a.s. Moreover, for $\dbP_\mu\in\cP$, it holds that $M^{\dbP} = M^{\dbP_0}-\int_0^\cd \mu_t\cd Z_tdt$. In particular, there exists $\dbP^*:=\dbP_{\mu^*}$ such that $M^{\dbP^*}$ is a $\dbP$-supermartingale for all $\dbP\in\cP$.
\end{lem}

We next make use of the Skorokhod decomposition in Section \ref{sec:skorokhod}. For the simplicity of notation, we denote $M^*:=M^{\dbP^*}$ and $A^*:=A^{\dbP^*}$. Consider the backward process:
\beaa
\l_t = (M^*_{T-t} - X_{T-t}) - (M^*_T- X_T).
\eeaa
Then we can find a non-negative process $\eta$ and a non-decreasing process $\k$ such that the statements in Lemma \ref{lem:skorokhod} holds. Denote the corresponding forward processes:
\beaa
\ol \eta_t :=\eta_{T-t} \q\mbox{and}\q \ol\k_t:=\k_{T-t}.
\eeaa

\begin{prop}\label{prop: k A}
It holds that
\beaa
\ol\k = A^*_T-A^*_\cd, \q\dbP_0\mbox{-a.s.}
\eeaa
\end{prop}
\begin{proof}
{\bf 1}.\q It follows from the Doob-Meyer decomposition \eqref{DM decom} that
\beaa
Y_t-X_t-Y_0+A^*_t-(M^*_T-X_T) = \l_{T-t} = \ol\eta_t -\ol\k_t,
\q\dbP_0\mbox{-a.s.}
\eeaa
Since $M^*_T-X_T=M^*_T-Y_T=A^*_T-Y_0$, $\dbP_0$-a.s., it holds
\beaa
(Y_t-X_t)-(A^*_T-A^*_t) = \ol\eta_t -\ol\k_t,
\q\dbP_0\mbox{-a.s.}.
\eeaa
Note that $Y\ge X$ and $A^*$ is non-decreasing, $\dbP_0$-a.s. By (iii) of Lemma \ref{lem:skorokhod}, we obtain
\be\label{k&a}
\ol\k \le A^*_T-A^*_\cd, \q \dbP_0\mbox{-a.s.}
\ee

\ms
\no {\bf 2}.\q Recall that
\beaa
\ol\k_t = -\min_{s\ge t}\Big((M^*_s-X_s)-(M^*_T-X_T)\Big).
\eeaa
Since $X_T = Y_T$, $\dbP_0$-a.s., it follows from \eqref{DM decom} that
\beaa
\ol\k_t = -\min_{s\ge t}(M^*_s-X_s)-A^*_T+Y_0,\q \dbP_0\mbox{-a.s.}
\eeaa
Taking nonlinear conditional expectation on both sides, we obtain
\be\label{est in step2}
\ul\cE_L[A^*_T-\ol\k_t|\cF_t] ~=~ Y_0-\ol\cE_L\big[\max_{s\ge t}(X_s-M^*_s)\big|\cF_t\big], \q \dbP_0\mbox{-a.s.}
\ee
Since by Lemma \ref{lem: DM} $M^*$ is $\dbP$-supermartingale for all $\dbP\in\cP$, we obtain
\beaa
\ol\cE_L\big[\max_{s\ge t}(X_s-M^*_s+M^*_t)\big|\cF_t\big]
\ge
\esup_{\t\in\cT_*^t}\ol\cE_L[(X_\t-M^*_\t+M^*_t)|\cF_t]
\ge 
\esup_{\t\in\cT_*^t}\ol\cE_L[X_\t|\cF_t]=Y_t,~\dbP_0\mbox{-a.s.}
\eeaa
In view of \eqref{k&a} and \eqref{est in step2}, we get
\beaa
A^*_t\le \ul\cE_L[A^*_T-\ol\k_t|\cF_t]\le Y_0 - Y_t +M^*_t=A^*_t,
\q \dbP_0\mbox{-a.s.}
\eeaa
It implies that $A^*_t =\ul\cE_L[A^*_T-\ol\k_t|\cF_t] $, $\dbP_0$-a.s. Again by \eqref{k&a}, we conclude that $A^*_t = A^*_T-\ol\k_t$, $\dbP_0$-a.s.
\end{proof}

\ms
\no {\bf Proof of Theorem \ref{thm: OS}}\q
We are going to prove that $\t^*:=\inf\{t:X_t=Y_t\}\in \cT_*$ is an optimal stopping time. By Lemma \ref{lem: DM} and Proposition \ref{prop: k A}, it holds
\beaa
Y_0 = Y_{\t^*}-M^*_{\t^*}+A^*_{\t^*}\q\mbox{and}\q
A^*_{\t^*}=\int_0^{\t^*}1_{\{t:X_t=Y_t\}}dA^*_t=0.
\eeaa
Therefore $Y_0 = \dbE^{\dbP^*}[Y_{\t^*}]$. Further, by (ii) of Lemma \ref{lem:skorokhod}, we may deduce that
\beaa
A,~Y\q\mbox{are both left continuous, $\dbP_0$-a.s.}
\eeaa
Hence $A^*_{\t^*}=0$, $\dbP_0$-a.s. Taking into account that $X$ is pathwise u.s.c., we obtain that
\beaa
Y_{\t^*}=X_{\t^*},\q \dbP_0\mbox{-a.s.}
\eeaa
Finally, we have
\beaa
Y_0 = \dbE^{\dbP^*}[Y_{\t^*}] = \dbE^{\dbP^*}[X_{\t^*}].
\eeaa
This implies that $\t^*$ is an optimal stopping time.
\qed

\section{Appendix}\label{sec: Appendix}

In preparation to the proof of Proposition \ref{prop: especial sol}, we study the processes:
\beaa
\bar \eta_t:=\ul\cE_L\big[u_\ch-\a\ch -\b B_\ch \big| \cF_t\big]
:=\einf_{\dbP\in\cP_L}\dbE^\dbP[u_\ch-\a\ch -\b B_\ch \big| \cF_t\big].
\eeaa
Similar to Proposition 6.5 in \cite{RTZ}, one may easily prove the following result of dynamic programming.

\begin{lem}\label{lem:eta hat}
There exists $Z\in\dbH^2$ such that
\beaa
\bar\eta_t = u_\ch -a(\ch-t) + \int_t^\ch L|\b-Z_s|ds -\int_t^\ch Z_s d B_s.
\eeaa
Moreover, it holds $\dbP_0\big[\eta_\t = \bar\eta_\t\big]=1$ for all $\t\in\cT_\ch$. In particular, we have
\be\label{eta DPP}
\eta_0 ~=~ \ul\cE_L\big[\eta_\t-\a\t-\b B_\t\big]\q \mbox{for all}~~\t\in\cT_\ch. 
\ee
\end{lem}

\no {\bf Proof of Proposition \ref{prop: especial sol}}\q
Without loss of generality, we only need to verify the properties at $\th=(0,0)$.

\no (i)\q By Lemma \ref{lem:eta hat}, $\eta$ is $\dbF^*$-adapted. Take $(\a',\b')\in\ul\cJ \eta_0$, i.e.
\beaa
\eta_0 = \max_{\t\in\cT_{\ch'}} \ol\cE_L\big[\eta_\t - \a'\t-\b' B_\t\big]~\mbox{for some}~\ch'\in\cT^+.
\eeaa
In view of \eqref{eta DPP}, we obtain that
\beaa
\dbE^{\dbP_\mu}\big[\eta_\t-\a\t-\b B_\t \big]\ge \eta_0\ge 
\dbE^{\dbP_\mu}\big[\eta_\t-\a'\t-\b' B_\t \big],
~~\mbox{for all}~\dbP_\mu \in\cP_L~\mbox{and}~\t\in\cT_{\ch\we\ch'}.
\eeaa
So, $\dbE^{\dbP_\mu} [-(\a'-\a)\t-(\b'-\b) B_\t]\le 0$ for all $\t\in\cT_{\ch\we\ch'}$. It follows that
\beaa
-\a'+\a-(\b'-\b)\cd\mu \le 0.
\eeaa
By taking $\mu^*:=-L\big(\sgn(\b'_i-\b_i)\big)_{1\le i\le d}$, we obtain that
\beaa
-\a'+\a+L|\b'-\b| \le 0.
\eeaa

\ms
\no (ii)\q  Since $u$ is Lipschitz continuous, one may easily estimate that
\beaa
|\eta(\th)-\eta(\th')| &\le & C \Big(\ol\cE_L\big[|\ch^\th -\ch^{\th'}|+\|B_{(\ch^\th-t)\we\cd}-B_{(\ch^{\th'}-t')\we\cd}\|\big]+d(\th,\th')\Big)\\
&\le & C' \Big(\ol\cE_L\big[|\ch^\th -\ch^{\th'}|\big]+d(\th,\th')\Big).
\eeaa
We applied BDG inequality for the last inequality. Since $\ch\in \cH$, we may suppose 
\beaa
\ch = T_0\we\ch_0,\q \ch_0:=\inf\{t:\o_t\notin O\}
~~\mbox{for some bounded open set}~O.
\eeaa
Then it is clear that $|\ch^\th -\ch^{\th'}|\le |t-t'|+|\ch^\th_0-\ch^{\th'}_0|$. Further it is proved in \cite{Ren-Ellip} that
\beaa
\lim_{d(\th,\th')\rightarrow 0} \ol\cE_L\big[|\ch^\th_0-\ch^{\th'}_0|\big] =0.
\eeaa
Therefore function $\eta$ is continuous.
\qed


\begin{thebibliography}{1}

\bibitem{BardiCapuzzo} Bardi, M. and Capuzzo-Dolcetta, I., {\it Optimal control and viscosity solutions of Hamilton-Jacobi-Belleman equations.} Birkhauser, 2008.

\bibitem{BSV} Beiglb\"ock, M., Schachermayer, W. and Veliyev, B., {\it A short Proof of the Doob-Meyer Theorem,} Stochastic Processes and Applications, 122 (2012), no. 4, 1204-1209.

\bibitem{CaffarelliCabre}
Caffarelli, L.A., and Cabre, X. {\sl Fully nonlinear elliptic equations}. American Mathematical Society Colloquium Publications, 43. American Mathematical Society, Providence, RI, 1995.

\bibitem{CSTV}
Cheridito, P., Soner, H.M., Touzi, N. and Victoir, N. {\it Second-order backward stochastic differential equations and fully nonlinear parabolic PDEs}. {\sl Comm. Pure Appl. Math.} 60 (2007), no. 7, 1081-1110.

\bibitem{CK}
Costantini, C. and Kurtz, T. G., {\it Viscosity methods giving uniqueness for martingale problems}, preprint.

\bibitem{CrandallLions}
Crandall, M.G. and Lions, P.-L.  {\it Viscosity solutions of Hamilton-Jacobi equations}. Trans. Amer. Math. Soc. 277 (1983), no. 1, 1-42.

\bibitem{CrandallIshiiLions}
Crandall, M.G., Ishii, H. and Lions, P.-L. {\it User's guide to viscosity solutions of second order partial differential equations}. {\sl Bull. Amer. Math. Soc. (N.S.)} 27 (1992), no. 1, 1-67.

\bibitem{Dupire}
Dupire, B. {\it Functional It\^{o} calculus}, {\sl papers.ssrn.com}, (2009).

\bibitem{EKTZ} Ekren, I., Keller, C., Touzi, N. and Zhang, J., {\it On Viscosity Solutions of Path Dependent PDEs}. Annals of Probability, 42 (2014), 204-236.

\bibitem{ETZ1}
Ekren, I., Touzi, N. and Zhang, J. {\it Fully Nonlinear Viscosity Solutions of Path Dependent PDEs: Part I}. Preprint,  arXiv:1210.0006.


\bibitem{FlemingSoner}
Fleming, W.H. and Soner, H.M. {\sl Controlled Markov processes and viscosity solutions}. Second edition. Stochastic Modelling and Applied Probability, 25. Springer, New York, 2006.

\bibitem{KQ}
Kobylanski, M. and Quenez, M., {\it Optimal stopping in a general framework.} Electronic Journal of Probability, Institute of Mathematical Statistics (IMS): OAJ, 17 (72), pp. 1-28, 2012.

\bibitem{PardouxPeng}
Pardoux, E. and Peng, S. G. {\it Adapted solution of a backward stochastic differential equation}. {\sl Systems Control Lett.} 14, no. 1, 55-61, 1990.

\bibitem{PX}
Peng, S. and Xu, M., {\it The smallest g-supermartingale and reflected BSDE with single and double $L^2$ obstacles.} Ann. I. H. Poincar\'e - PR 41, 605-630, 2005. 

\bibitem{PhamZhang}
Pham, T. and Zhang, J. {\it Two Person Zero-Sum Game in Weak Formulation and Path Dependent BellmanÐIsaacs Equation}. {\sl SIAM J. Control Optim.}  52 (2014), no. 4, 2090-2121.


\bibitem{Ren-Ellip} Ren, Z., {\it Viscosity Solutions of Fully Nonlinear Elliptic Path Dependent Partial Differential Equations}, preprint.

\bibitem{RTZ-survey} Ren, Z., Touzi, N. and Zhang, J., {\it An Overview of Viscosity Solutions of Path Dependent PDEs,} Stochastic Analysis and Applications 2014, Springer Proceedings in Mathematics \& Statistics Volume 100, 2014, pp 397-453.

\bibitem{RTZ} Ren, Z., Touzi, N. and Zhang, J., {\it Comparison of Viscosity Solutions of Semi-linear Path-Dependent PDEs,} preprint.


\bibitem{STZ2}
Soner, H.M., Touzi, N. and Zhang, J. {\it Wellposedness of second order backward SDEs}. {\sl Probab. Theory Related Fields} 153 (2012), no. 1-2, 149-190.


\end{thebibliography}
\end{document}